\documentclass[12pt,a4paper]{amsart}
\usepackage{amsfonts}
\numberwithin{equation}{section}

\theoremstyle{plain}
\newtheorem{Th}{Theorem}[section]

 \theoremstyle{definition}

\newtheorem{?}[Th]{Problem}

\renewcommand{\Re}{\operatorname{Re}}
  
  \newcommand{\ab}[1]{\left\vert{#1}\right\vert}
 \newcommand{\zj}[1]{\left({#1}\right)}

\newcommand{\eq}[1]{\eqref{#1}}

\begin{document}

\title{Average Goldbach and the Quasi-Riemann hypothesis
}
\author{Gautami Bhowmik}
\address{Laboratoire Paul Painlev\'e\\
Labex-cempi, Universit{\'e} Lille 1, 59655
Villeneuve d'Ascq Cedex, France
}
\email{bhowmik@math.univ-lille1.fr}

\author{Imre Z. Ruzsa}
\address{Alfr\'ed R\'enyi Institute of Mathematics\\
     Budapest, Pf. 127\\
     H-1364 Hungary
}
\email{ruzsa@renyi.hu}
 \thanks{The second author was supported by ERC--AdG Grant No.321104  and
 Hungarian National Foundation for Scientific
 Research (OTKA), Grants No.109789  
 and NK104183.}           
\keywords{Goldbach problem, Riemann hypothesis,
Chebyshev Function, Dirichlet Kernel}

 \subjclass[2010]{11P32, 11M26}
    \begin{abstract}
    We prove that a good average order on the Goldbach generating function implies that the real parts of the non-trivial zeros of the Riemann zeta function
    are strictly less than 1. This together with existing results establishes an equivalence between such asymptotics and the Riemann Hypothesis.
     \end{abstract}

     \maketitle

\section{Introduction} The  relation between the average order of the Goldbach generating function and the Riemann Hypothesis was introduced by Granville \cite{Gra07}.
In an attempt to generalise such relations to sums of primes in arithmetic progressions it was observed that  asymptotic results on the average number of restricted
representations of integers   assuming a conjecture on distinct zeros of Dirichlet L-functions 
did not obviously yield an equivalence with the generalised Riemann Hypothesis since the case of the real part of non-trivial zeros being exactly 1 could not 
be eliminated. In this note we succeed in doing so for the Riemann zeta function in a simple way and establish an equivalence of average orders of the original Goldbach function with what we call a quasi-Riemann hypothesis, that the real parts of the non-trivial zeros of of the Riemann zeta function 
are    strictly less than $1$. This completes the proof of Theorem 1A in \cite{Gra07} i.e. a good asymptotic order 
of the Goldbach function is indeed equivalent to the Riemann Hypothesis. The same idea has since been applied  to the restricted sums of two primes both in the same congruence class (Section 8 \cite{BHMS}).

The average order of the function \[ G(n) = \sum_{k_1+k_2=n} \Lambda(k_1) \Lambda(k_2) , \] where $\Lambda$ is the von Mangoldt function 
was studied by Fujii \cite{Fuj91} and Granville (op.cit) who showed that under the Riemann Hypothesis  the summatory function can be expressed as
\begin{equation}
 S(x) = \sum_{n\leq x} G(n)
        = \frac{x^2}{2} + O(x^{3/2}) . \end{equation}
The error term was  eventually improved to its conjectured value $O(x^{1+\varepsilon})$ \cite{B_SP} but here it is enough for us to assume any error term less than $x^2$ and from the square root of the associated power series  estimate well the  Chebyshev function,
$\psi(x)=\sum_{n\leq x}\Lambda(n)$,   using a  Dirichlet kernel.

\section {Quasi-Riemann Hypothesis}
  \begin{Th}
  If the asymptotic relation 
  \begin{equation}\label{ass}
        S(x) = \frac{x^2}{2} + O(x^{2-\delta}), \ \delta>0  
        \end{equation}
holds, then  there exists $0< \delta'  <1  $ such that 
 for  the non-trivial zeros $\rho$ of the Riemann zeta function $ \Re \rho < 1- \delta'$.
\end{Th}

\begin{proof}
We consider the power series
 \[ F(z) = \sum _{n=1}^\infty \Lambda(n) z^n, \ |z| <1  \]
whose square is written in terms of the Goldbach function as
  \[ F(z)^2 = \sum_{n=1}^\infty G(n)z^n .\]
 Writing the above as a telescoping series with $G(n)=S(n)-S(n-1)$ we get $$F(z)^2= (1-z) \sum_{n=1}^\infty S(n)z^n.$$
  From the assumption \eq{ass} on the Goldbach asymptotic order we infer
 
   \begin{equation}\label{asym_sum}
   \sum_{n=1}^\infty  S(n)z^n = \sum_n \left(\frac{n^2}{2} + O(n^{2-\delta}) \right)z^n
   =\sum_n\frac{n^2}{2}z^n+O\left(\sum_n n^{2-\delta}|z|^n\right).
   \end{equation}

   We have
    \[ \sum_n \frac{n^2}{2} z^n =  \frac{1}{(1-z)^3} - \frac{3}{2(1-z)^2} +\frac{1}{2(1-z)} ,\] 
 while the error term of (\ref{asym_sum} ) is estimated as
$$\sum_{n=1}^\infty n^{2-\delta}|z|^n \ll (1-|z|)^{\delta-3}  $$ 
by comparison with the power series expansion of the right hand side 
and thus
 $$\sum_n S(n)z^n= \frac{1}{(1-z)^3} +O\left(\frac{1}{(1-|z|)^{3-\delta}  }\right) $$
which implies
 $$F(z)^2= (1-z) \sum_{n=1}^\infty S(n)z^n =  \frac{1}{(1-z)^2}
    + O\zj{ \frac{|1-z|}{(1- |z| )^{3-\delta} }}.$$

To estimate $\psi(N)$ we consider the circle $|z| =R= 1-1/N$ for a large positive integer $N$ and rewrite the last formula as
\begin{equation}\label{square}
F(z)^2 = \frac{1}{(1-z)^2} + O\bigl( |1-z| N^{3-\delta} \bigr).
\end{equation}

The error term is less than the absolute value of the  main term for
\begin{equation}\label{major}  |1-z| <  N^ {\frac{\delta}{3}-1}, \end{equation}
and in this range we take the complex square root of \eq{square} to obtain
\begin{equation}\label{root}
F(z) = \frac{1}{1-z} + O\bigl( |1-z|^2 N^{3-\delta} \bigr).
\end{equation}

\underline {Kernel.} We introduce the function
$$K(z)=z^{-N-1}\frac{1-z^N}{1-z}.$$
We have $K(z)\ll |1-z|^{-1}$ and  we  use it to estimate the partial sums
 \begin{equation}\label{kernel}
 \psi(N) = \frac{1}{2\pi i} \int_{|z|=R}  F(z)K(z) dz \\
= N + \frac{1}{2\pi i} \int_{|z|=R}\left(F(z)-\frac{1}{1-z}\right)K(z) dz . 
  \end{equation}
  
\underline{Major Arc.} On the part of the circle \eq{major} the integrand is
\[  O\bigl( |1-z| N^{3-\delta} \bigr) = O\bigl(  N^{2-2\delta/3} \bigr)\]
so the contribution to the integral above is $ O\bigl(  N^{1-\delta/3} \bigr)$
from the power series estimate obtained in (\ref{root}).

\smallskip
\underline {Minor Arc.} On the rest of the circle, i.e. the arc $|1-z|>N^{\delta/3-1}$ we use the Cauchy-Schwarz inequality.
For $F$ we apply 
the  estimate for the full circle

\begin{equation*}
\begin{split}
\int_{|z|=R} \ab{F(z)-\frac{1}{1-z}}^2 |dz|
&=\int_0^{2\pi}\sum_n(\Lambda(n)-1)R^ne^{int}\sum_m(\Lambda(n)-1)R^ne^{-imt}\\
&=2\pi \sum_{n=1}^\infty (\Lambda(n)-1)^2 \zj{1-\frac{1}{N}}^{2n} = O(N \log N).
\end{split}
\end{equation*}
We estimate the {\it square integral of the kernel} as follows.
\[
I= \int_{\substack{|z|=R \\ |1-z|>N^{\delta/3-1}}} |K(z)|^{2} |dz| \ll 
   \int_{\substack{|z|=R \\ |1-z|>  N^{\delta/3-1} } }  \frac{ |dz| }{|1-z|^{2}}.
\]
With the parametrization $ z=Re^{it}$ the last integral becomes
 \[ 2\int_{t_0}^\pi \frac{1}{ |1-Re^{it} |^2} dt, \]
where $t_0$ is defined by $ |1-Re^{it_0} | =  N^{\delta/3-1} $.

We have
 \[  |1-Re^{it} |^2 = 1 + R^2 - 2R \cos t = (1-R)^2 + 4R^2 \sin^2 \frac{t}{2} \
\begin{cases} < t^2 + N^{-2} \\ > t^2/3, \end{cases}  \]
whence $t_0 \gg  N^{\delta/3-1} $ and the integral satisfies
 \[ I \ll \int_{t_0}^\pi t^{-2} \, dt < t_0^{-1} \ll N^{1-\delta/3} . \]

Now the Cauchy Schwarz inequality gives the bound for the
 minor arc as
\[
  \Big( \int_{|z|=R} \Big|F(z)-\frac{1}{(1-z)}\Big|^{2}dz
  \Big)^{1/2} I^{1/2} 
  \ll  N^{1-\delta/6} ( \log N)^{1/2} .
\]
 The major and minor arcs together give
\[
  \psi(N)-N \ll  N^{1-\delta/6} ( \log N)^{1/2} ,
\]
which proves the theorem with $\delta'=\delta/6$.

   \end{proof}
\section{Remarks}
 1. Using a better kernel and a more careful calculation we can improve the bound to $\delta'=\delta/3$. It is of little importance in this context 
as the methods of [1,2,4] give $\delta'=\delta$ as soon as we have $\delta'<1$, through an explicit expression of $S(x)$ using roots of the zeta-function.
 We are unable to get $\delta'=\delta$ directly by the simple method of this paper and we think it cannot be done. We used very little of the properties
 of the Liouville function and our method works equally for any function $f$ satisfying $ |f(n)| = O(n^\varepsilon)$ for all $\varepsilon>0$ and
  \[ \sum_{m+n\leq x} f(m)f(n) = cx^2 + O(x^{2-\delta}) .   \]
We doubt that
\[ \sum_{n\leq x} f(n) = c'x + O(x^{1-\delta+o(1)})  \]
follows in this generality.

2. In general we are yet unable to establish an equivalence between average orders of
restricted Goldbach functions and  appropriate quasi Riemann Hypotheses for $L$-functions of \cite{Gra07} or \cite{BHMS} though a particular case is handled
in [1] using the method of this note.

     \end{document}